\documentclass[10pt]{article}
\usepackage[utf8]{inputenc}

\usepackage{amssymb,amsthm,amsmath}
\usepackage{enumerate}
\usepackage{graphicx,color}
\usepackage[hidelinks]{hyperref}

\newcommand{\dd}{\mathrm{d}}
\newcommand{\E}{\mathbb{E}}
\newcommand{\1}{\textbf{1}}
\newcommand{\R}{\mathbb{R}}

\newcommand{\scal}[2]{\!\left\langle #1, #2 \right\rangle\!}

\renewcommand{\leq}{\leqslant}
\renewcommand{\geq}{\geqslant}
\newcommand{\ls}{\leqslant}
\newcommand{\gr}{\geqslant}

\DeclareMathOperator{\Vol}{Vol}

\usepackage[paper=a4paper, left=1.5in, right=1.5in, top=1.6in, bottom=1.4in]{geometry}
\newtheorem{theorem}{Theorem}
\newtheorem{lemma}[theorem]{Lemma}
\newtheorem{corollary}[theorem]{Corollary}

\theoremstyle{remark}
\newtheorem{remark}[theorem]{Remark}

\theoremstyle{definition}

\title{\vspace{-3em}From Ball's cube slicing inequality to Khinchin-type inequalities for negative moments}

\author{
Giorgos Chasapis\thanks{\linespread{1.0} Carnegie Mellon University; Pittsburgh, PA 15213, USA. Email: gchasapi@andrew.cmu.edu.}, \
Hermann K\"onig\thanks{Mathematisches Seminar,\,Universit\"at Kiel,\,24098 Kiel,\,Germany.\,E-mail:\,hkoenig@math.uni-kiel.de.}  
\ and
Tomasz Tkocz\thanks{\linespread{1.0} Carnegie Mellon University; Pittsburgh, PA 15213, USA. Email: ttkocz@math.cmu.edu. Research supported in part by NSF grant DMS-1955175} \\
\date{24th November 2020}
}

\begin{document}

\maketitle

\begin{abstract}
We establish a sharp moment comparison inequality between an arbitrary negative moment and the second moment for sums of independent uniform random variables, which extends Ball's cube slicing inequality.
\end{abstract}

\bigskip

\begin{footnotesize}

\end{footnotesize}

\bigskip

\section{Introduction}

Ball's celebrated cube slicing inequality established in \cite{Ball} states that the maximal volume cross-section of the centred cube $[-1, 1]^n$ in $\R^n$ by a hyperplane (a subspace of codimension $1$) equals $2^{n-1}\sqrt{2}$, attained by the hyperplane with normal vector $(\frac{1}{\sqrt{2}}, \frac{1}{\sqrt{2}}, 0, \dots, 0)$ (see also \cite{Ball2}). Khinchin-type inequalities provide moment comparison, typically for weighted sums of independent identically distributed (i.i.d.) random variables. The classical one concerns symmetric random signs and goes back to the work \cite{Kh} of Khinchin. Such inequalities are instrumental in studying unconditional convergence and are used extensively in (functional) analysis and geometry, particularly in (local) theory of Banach spaces. We refer to several works \cite{BC, Haa, K, KK, KLO, LO-best, NO, NP, Sa, Sz} for further background and references (particularly, \cite{BC} provides a detailed historic account on Khinchin inequalities with sharp constants).

The main motivation for this article and its starting point is a fact well-known to experts that Ball's inequality can be viewed as a Khinchin-type inequality (the dual question of extremal volume hyperplane-projections of convex bodies is also linked to Khinchin-type inequalities, see for example \cite{Ball3, BN, ENT1}). An elementary derivation can be sketched as follows. For a unit vector $a = (a_1, \dots, a_n)$ in $\R^n$, let $f$ be the density of $X = \sum_{k=1}^n a_kU_k$, where $U_1, \dots, U_n$ are i.i.d. uniform on $[-1,1]$. Then the $(n-1)$-volume of the cross-section of the cube $[-1,1]^n$ by the hyperplane $a^\perp$ perpendicular to $a$ is $\Vol_{n-1}\big([-1,1]^n \cap a^\perp\big) = 2^nf(0)$. On the other hand, for every symmetric unimodal bounded random variable $X$ with density $f$, we have
\[
f(0) = \|f\|_\infty = \lim_{p \to 1-} \frac{1-p}{2}\E|X|^{-p}
\]
($X$ is called symmetric if it has the same distribution as $-X$).
Thus Ball's inequality, put probabilistically, says that for every unit vector $a$ in $\R^n$, we have
\[
\lim_{p \to 1-} (1-p)\E\left|\sum_{k=1}^n a_kU_k\right|^{-p} \leq \sqrt{2}.
\]
Our main result shows in particular that not only does this inequality hold in the limit, but also for every $p \in (p_0,1)$, where $p_0 = 0.793...\,$. To view this inequality as actual moment comparison, let $\xi_1, \xi_2, \dots$ be i.i.d. random vectors in $\R^3$ uniform on the centered Euclidean unit sphere $S^2$. As a result of Archimedes' hat-box theorem and rotational invariance, the left hand side can be rewritten as $\E\left\|\sum_{k=1}^n a_k\xi_k\right\|^{-1}$, where $\|\cdot\|$ stands for the standard Euclidean norm on $\R^3$ (see Lemma \ref{lm:konig-identity} below). We thus have the following identity for a unit vector $a$ in $\R^n$,
\begin{equation}\label{eq:vol-identity}
2^{1-n}\Vol_{n-1}\big([-1,1]^n \cap a^\perp\big) =  \lim_{p \to 1-} (1-p)\E\left|\sum_{k=1}^n a_kU_k\right|^{-p} = \E\left\|\sum_{k=1}^n a_k\xi_k\right\|^{-1}.
\end{equation}
For a generalisation, see Proposition 3.2 in \cite{KR}.
As a result, we can rephrase Ball's inequality as the following sharp $L_{-1}-L_2$ Khinchin-type inequality: for every $n$ and every reals $a_1, \dots, a_n$,
\begin{equation}\label{eq:L-1-1}
\E\left\|\sum_{k=1}^n a_k\xi_k\right\|^{-1} \leq \sqrt{2}\left(\sum_{k=1}^n a_k^2\right)^{-1/2}.
\end{equation}
We extend this to a sharp $L_{-p}-L_2$ moment comparison for $p \in (0,1)$ with arbitrary matrix-valued coefficients (Corollary \ref{cor:khin-sph} below). We refer to \cite{BC, K, KK, LO} for sharp results for positive moments.

We describe our results in the next section and then present our proofs, preceded with a short overview of them. We conclude with a summary highlighting possible future work.
Throughout, $\scal{x}{y} = \sum_{j=1}^d x_jy_j$ denotes the standard scalar product on $\R^d$, $\|x\| = \sqrt{\scal{x}{x}}$ is the Euclidean norm whose unit sphere and closed unit ball are denoted by $S^{d-1}$ and $B_2^d$, respectively. Moreover, $e_j$ is the $j$-th vector of the standard basis whose $j$-th coordinate is $1$ and the rest are $0$.

\section{Results}

Let $U_1, U_2$ be i.i.d. random variables uniform on $[-1,1]$ and let $Z$ be a standard Gaussian random variable (mean $0$, variance $1$). For $p \in (0,1)$, we define the constants
\begin{equation}\label{eq:consts}
\begin{split}
c_2(p) &= \E\left|\frac{U_1+U_2}{\sqrt{2}}\right|^{-p} = 2^{p/2}\int_{-2}^2 |x|^{-p}\left(\frac{1}{2}-\frac{|x|}{4}\right)\dd x = \frac{2^{1-p/2}}{(1-p)(2-p)}, \\
c_\infty(p) &= \E\left|\frac{Z}{\sqrt{3}}\right|^{-p} = \frac{3^{p/2}}{\sqrt{2\pi}}\int_{-\infty}^\infty |x|^{-p}e^{-x^2/2}\dd x= \frac{(3/2)^{p/2}}{\sqrt{\pi}}\Gamma\left(\frac{1-p}{2}\right)
\end{split}
\end{equation}
and
\begin{equation}\label{eq:Cp}
C_p = \max\{c_2(p),c_\infty(p)\}.
\end{equation}
By comparing $c_2(p)$ and $c_\infty(p)$ as done in Lemma \ref{lm:c2-cinf} from Section \ref{sec:tech-lm} below, in fact we have
\[
C_p = \begin{cases} c_\infty(p), & \text{if } p \in (0,p_0), \\ c_2(p), & \text{if } p \in (p_0,1),\end{cases}
\]
where $p_0$ is the unique $p\in(0,1)$ such that $c_2(p)=c_\infty(p)$. Our main result is the following $L_{-p}-L_2$ Khinchin-type inequality for sums of symmetric uniform random variables.

\begin{theorem}\label{thm:khin-unif}
Let $p \in (0,1)$ and let $C_p$ be defined by \eqref{eq:Cp}. Let $U_1, U_2, \ldots$ be i.i.d. random variables uniform on $[-1,1]$. For every $n$ and every reals $a_1, \dots, a_n$, we have
\begin{equation}\label{eq:khin-unif}
\E\left|\sum_{k=1}^n a_kU_k\right|^{-p} \leq C_p\left(\sum_{k=1}^n a_k^2\right)^{-p/2}.
\end{equation}
\end{theorem}

\begin{remark}
Applying \eqref{eq:khin-unif} to $n=2$, $a_1 = a_2 = \frac{1}{\sqrt{2}}$ and to $n$ large, $a_1 = \dots = a_n = \frac{1}{\sqrt{n}}$ (with the aid of the central limit theorem) shows that the value of $C_p$ in \eqref{eq:khin-unif} is sharp.
\end{remark}

Moments of a Euclidean norm of weighted sums of independent random vectors uniform on $S^{d+1}$ and $B_2^{d}$, $d \geq 1$, are proportional (see Proposition 4 in \cite{KK} or its generalisation, Theorem 4 in \cite{BC}). We recall a special case of this result relevant for us and for convenience sketch its proof (particularly because the proofs available in the literature treat the case of positive moments, but of course they repeat verbatim to negative moments).

\begin{lemma}[Proposition 4, \cite{KK}]\label{lm:konig-identity}
Let $\xi_1, \xi_2, \ldots$ be i.i.d. random vectors uniformly distributed on the unit sphere $S^{2}$ in $\R^{3}$. Let $U_1, U_2, \ldots$ be i.i.d. random variables uniform on $[-1,1]$. For a vector $a = (a_1,\ldots,a_n)$ in $\R^n$ and $p \in (-\infty,1)$, we have
\[
\E\left\|\sum_{k=1}^n a_k\xi_k\right\|^{-p} = (1-p)\E\left|\sum_{k=1}^n a_kU_k\right|^{-p}.
\]
\end{lemma}
\begin{proof}
We reproduce here an argument utilising rotational invariance from \cite{KK} attributed to Lata\l a. Let $\theta$ be a random vector uniform on $S^2$, independent of all the other variables. By rotational invariance, for a vector $x$ in $\R^3$, we have
\[
\E|\scal{x}{\theta}|^{-p} = \E|\scal{e_1\|x\|}{\theta}|^{-p} = \|x\|^{-p}\E|\theta_1|^{-p},
\]
where $\theta_1$ denotes the first component of $\theta$, so
\[
\|x\|^{-p} = \frac{\E|\scal{x}{\theta}|^{-p}}{\E|\theta_1|^{-p}}.
\]
Applying this to $x = \sum_{k=1}^n a_k\xi_k$ and taking the expectation gives
\[
\E_\xi\left\|\sum_{k=1}^n a_k\xi_k\right\|^{-p} = \E_\xi\E_\theta\frac{|\scal{\sum_{k=1}^n a_k\xi_k}{\theta}|^{-p}}{\E|\theta_1|^{-p}} = \frac{1}{\E|\theta_1|^{-p}}\E_\theta\E_\xi\left|\scal{\sum_{k=1}^n a_k\xi_k}{\theta}\right|^{-p}.
\]
By the rotational invariance of $\sum a_k\xi_k$, we also have
\[
\E_\xi\left|\scal{\sum_{k=1}^n a_k\xi_k}{\theta}\right|^{-p} = \E_\xi\left|\scal{\sum _{k=1}^na_k\xi_k}{e_1\|\theta\|}\right|^{-p}.
\]
However, $\theta$ is a unit vector and the random variables $\scal{\xi_k}{e_1}$ are i.i.d. uniform on $[-1,1]$, therefore
\[
\E_\xi\left|\scal{\sum_{k=1}^n a_k\xi_k}{e_1\|\theta\|}\right|^{-p} = \E_\xi\left|\sum_{k=1}^n a_k\scal{\xi_k}{e_1}\right|^{-p} = \E_U\left|\sum_{k=1}^n a_kU_k\right|^{-p}.
\]
Since $\theta_1$ is uniform on $[-1,1]$, we get $\E|\theta_1|^{-p} = \int_0^1 x^{-p} \dd x = \frac{1}{1-p}$. Putting these equations together finishes the proof.
\end{proof}

It follows from Lemma \ref{lm:konig-identity} that \eqref{eq:khin-unif} is equivalent to
\begin{equation}\label{eq:khin-sph-scal}
\E\left\|\sum_{k=1}^n a_k\xi_k\right\|^{-p} \leq (1-p)C_p\left(\sum_{k=1}^n a_k^2\right)^{-p/2}.
\end{equation}
We extend this to matrix-valued coefficients using isometrical embeddings into $L_p$ spaces (Orlicz-Szarek's argument, see Remark 3 in \cite{Sz}). This offers a sharp version of the very general result of Gorin and Favarov from \cite{GF} (see Corollary 2 therein) in the case of uniform vectors on $S^{2}$ and the $L_{-p}-L_2$ moment comparison. For a matrix $A $, $\|A\|_{HS}$ stands for its Hilbert-Schmidt norm.

\begin{corollary}\label{cor:khin-sph}
Let $p \in (0,1)$ and let $C_p$ be defined by \eqref{eq:Cp}. Let $\xi_1, \xi_2, \ldots$ be i.i.d. random vectors uniform on the unit sphere $S^2$ in $\R^3$. For every $n$ and every real $3 \times 3$ matrices $A_1, \dots, A_n$, we have
\begin{equation}\label{eq:khin-sph}
\E\left\|\sum_{k=1}^n A_k\xi_k\right\|^{-p} \leq C_p\left(\sum_{k=1}^n \|A_k\|_{HS}^2\right)^{-p/2}.
\end{equation}
\end{corollary}

\begin{remark}
Both \eqref{eq:khin-sph-scal} and \eqref{eq:khin-sph} are sharp. The constant in \eqref{eq:khin-sph} is larger than in \eqref{eq:khin-sph-scal} . The former specialised to the case when each matrix $A_k$ is proportional to the matrix $\left[\begin{smallmatrix} 1 & 0 & 0 \\ 0 & 0 & 0\\ 0 & 0 & 0\end{smallmatrix}\right]$ reduces to \eqref{eq:khin-unif}.
\end{remark}

\begin{remark}
A sharp reversal of \eqref{eq:khin-sph-scal} (analogously of \eqref{eq:khin-sph}) is immediate from convexity,
\[
\E\left\|\sum_{k=1}^n a_k\xi_k\right\|^{-p} \geq \left(\E\left\|\sum_{k=1}^n a_k\xi_k\right\|\right)^{-p} \geq \left(\left(\E\left\|\sum_{k=1}^n a_k\xi_k\right\|^2\right)^{1/2}\right)^{-p} = \left(\sum_{k=1}^n a_k^2\right)^{-p/2}.
\]
By \eqref{eq:vol-identity}, the case $p=1$ of this inequality gives yet another simple proof of Hadwiger's and Hensley's result (see \cite{Ha} and \cite{H}, see also Theorem 2 in \cite{Ball}).
\end{remark}

\section{Proof overview}

Haagerup's work \cite{Haa}  can perhaps be seen as a landmark in the pursuit of sharp Khinchin-type inequalities. Later, Nazarov and Podkorytov in \cite{NP} offered an informative exposition of \cite{Haa} (and \cite{Ball}), developing novel tools which allowed for significant simplications of the most technically demanding parts of \cite{Haa} (as well as of \cite{Ball}). We shall closely follow their approach which comprises two main steps. (For other works which used techniques from \cite{NP} to establish sharp Khinchin-type inequalities, we refer for instance to \cite{K, Mo}.)

\emph{Step I (Section \ref{sec:small}).} We prove \eqref{eq:khin-unif} in the case that all weights $a_k$ are ``small'', that is for the sequences $a = (a_k)_{k=1}^n$ with $\max_{k \leq n} |a_k| \leq \frac{1}{\sqrt{2}}\left(\sum_{k=1}^n a_k^2\right)^{1/2}$ (call it Case A). This in turn is accomplished by a Fourier-analytic expression for negative moments (used for instance in \cite{GF}), which allows to leverage independence. As in \cite{Ball}, by the use of H\"older's inequality, the following integral inequality allows to finish the whole argument,
\begin{equation}\label{eq:ball-int}
s^{p/2}\int_0^\infty \left|\frac{\sin t}{t}\right|^s t^{p-1} \dd t \leq 2^{p-1}\sqrt{\pi}\frac{\Gamma\left(\frac{p}{2}\right)}{\Gamma\left(\frac{1-p}{2}\right)}C_p, \qquad 0 < p < 1, s \geq 2.
\end{equation}
This inequality is an extension of Ball's integral inequality from \cite{Ball} and is proved with the methods of \cite{NP}. For other refinements and extensions of Ball's integral and cube slicing inequalities see \cite{EM, KOS, KKol, KKol2, LPP}.

\emph{Step II (Section \ref{sec:large}).} With the aid of the result of Step I, we use induction on $n$ to prove a certain strengthening of \eqref{eq:khin-unif} for \emph{all} sequences $a = (a_k)_{k=1}^n$ in order to handle those which do not satisfy Case A, that is have a ``large'' weight (call those Case B). Were \eqref{eq:ball-int} true for all $s \geq 1$, this step would have been spared. In \cite{NP} the inductive step is possible thanks to an algebraic identity obtained by averaging with respect to one random sign. In our setting, for uniform $[-1,1]$ random variables, such an identity does not seem to present itself. To overcome this obstacle, we work with $S^2$-uniform random vectors for which certain algebraic identities allowing for induction are much more natural. For Ball's inequality \eqref{eq:L-1-1} (case $p = 1$), this step was in \cite{Ball} taken care of by a simple projection argument, but its analogue for $p<1$ is not sufficient (see Remark \ref{rem:proj} at the end of Section \ref{sec:large}).

We remark that in the range $p \in (0, p_0)$ when $C_p = c_\infty(p)$ and the extremizing sequence is $a_1 = \dots = a_n = \frac{1}{\sqrt{n}}$ with $n\to \infty$, it is only Case A which admits equality (attained asymptotically as $n\to \infty$), whereas in the range $p \in (p_0,1)$ when $C_p = c_2(p)$ and the extremizing sequence is $a_1 = a_2 = \frac{1}{\sqrt{2}}$, $n = 2$, both Case A and B admit equality (in Case B when taking $n=2$ and $a_1 = \frac{1}{\sqrt{2-\delta}}$, $a_2 = \frac{1}{\sqrt{2+\delta}}$, $\delta \to 0+$) and hence both Step I and II have to be subtle enough to overcome this difficulty.

As a final comment here, convexity-type arguments leading to more precise results such as Schur-convexity of moments of sums with a \emph{fixed} number of summands $n$ (see \cite{AH, BC, Eat, ENT1, FHJSZ, HT, Kom, KK, Pe}) do not seem to be available here. One of the obstacles is for instance the fact that the function $t \mapsto \E|U_1 + \sqrt{t}U_2|^{-p}$ is not convex/concave on the whole half-line $(0,+\infty)$ (it is concave on $(0,1)$ and convex on $(1,+\infty)$).

\section{Technical lemmas}\label{sec:tech-lm}

We gather several elementary but technical results needed in our proofs. The first one explains the comparison between the constants $c_2(p)$ and $c_\infty(p)$ arising from two different extremizing sequences of weights $a_k$ in our Khinchin inequality.

\begin{lemma}\label{lm:c2-cinf}
Let $c_2(p), c_\infty(p)$ be defined in \eqref{eq:consts}. The equation $c_2(p) = c_\infty(p)$ has a unique solution $p_0 = 0.793...$ on $(0,1)$. Moreover, $c_2(p) > c_\infty(p)$ for $p \in (p_0,1)$, whilst $c_2(p) < c_\infty(p)$ for $p \in (0,p_0)$.
\end{lemma}
\begin{proof}
For $p \in (0,1)$, the difference $c_2(p) - c_\infty(p)$ has the same sign as
\[
f(p) = 2\sqrt{\pi}3^{-p/2} - (1-p)(2-p)\Gamma\left(\frac{1-p}{2}\right).
\]
\textbf{Claim.} The function $p \mapsto \log\left((1-p)(2-p)\Gamma\left(\frac{1-p}{2}\right)\right)$ is strictly concave on $(0,1)$.

Note that $f(0+) = 0$, $f(1-) = 2(\sqrt{\pi/3}-1) > 0$  ($u\Gamma(u) \to 1$ as $u \to 0$) and $f(\frac{2}{3}) = 2\sqrt{\pi}3^{-1/3}-\frac{4}{9}\Gamma(\frac{1}{6}) < -0.016$. In view of the claim (after taking the logarithm and noting that a linear function intersects a strictly concave function at most twice), the proof of the lemma is finished. 

To prove the claim, we let $u = \frac{1-p}{2}$ and $h(u) = -\log\left(u\left(u+\frac{1}{2}\right)\Gamma(u)\right)$. We want to show that $h$ is strictly convex on $(0,\frac12)$. Recall $(\log \Gamma(u))'' = \sum_{k=0}^\infty \frac{1}{(u+k)^2}$. Thus for $u \in (0,\frac{1}{2})$,
\begin{align*}
h''(u) &= \frac{1}{u^2} + \frac{1}{(u+\frac{1}{2})^2} - \sum_{k=0}^\infty \frac{1}{(u+k)^2}
> \frac{1}{(u+\frac{1}{2})^2} - \frac{1}{(u+1)^2} - \frac{1}{(u+2)^2} - \sum_{k=3}^\infty \frac{1}{k^2}.
\end{align*}
We now show that the right hand side is positive on $(0,\frac{1}{2})$. Call it $h_1(u)$ and note that
\[
\frac12h_1'(u) = \frac{1}{(u+1)^3} + \frac{1}{(u+2)^3} - \frac{1}{(u+\frac12)^3} < \frac{2}{(u+1)^3} - \frac{1}{(u+\frac12)^3}.
\]
The right hand side has the same sign as $\left(\frac{u+\frac12}{u+1}\right)^3-\frac{1}{2}$ which is clearly increasing in $u$, thus at most $\left(\frac{2}{3}\right)^3 - \frac{1}{2} = -\frac{11}{54} < 0$ for $u \in (0,\frac12)$. Thus $h_1(u)$ is decreasing for $u \in (0,\frac12)$. Going back to the lower bound $h''(u) > h_1(u)$, we conclude that $h''(u) > h_1(u) > h_1(\frac{1}{2}) = \frac{1481}{900} - \frac{\pi^2}{6} > 0.0006$, for $u \in (0,\frac{1}{2})$. This shows that $h$ is strictly convex on $(0,\frac{1}{2})$.
\end{proof}

The next three lemmas are elementary facts about functions showing up in calculations from Step I (Section \ref{sec:small}) needed to prove the integral inequality \eqref{eq:ball-int}.

\begin{lemma}\label{lm:cos-sin}
$\sup_{t \in \R}\left|\cos t - \frac{\sin t}{t}\right| < \frac{6}{5}$.
\end{lemma}
\begin{proof}
Since both $\cos t$ and $\frac{\sin t}{t}$ are even, it suffices to consider positive $t$. By the Cauchy-Schwarz inequality, we have $\left|\cos t - \frac{\sin t}{t}\right| \leq \sqrt{1+\frac{1}{t^2}}$, so it suffices to consider $t \leq \frac{5}{\sqrt{11}}$. It remains to note that $\frac{5}{\sqrt{11}} < \frac{\pi}{2}$ and that on $(0, \frac{\pi}{2})$, we have $\left|\cos t - \frac{\sin t}{t}\right| = \frac{\sin t}{t} - \cos t < 1 +  0 = 1$.
\end{proof}

\begin{lemma}\label{lm:t0}
Let $y_1 = \max_{t \in [\pi, 2\pi]} \left|\frac{\sin t}{t}\right|$. For $y \in (0,y_1)$, let $t = t_0$ be the unique solution to $\frac{\sin t}{t} = y$ on $(0,\pi)$. Then $t_0 > 2$.
\end{lemma}
\begin{proof}
Since $y_1 < \frac{1}{\pi}$, for every $y\in(0,y_1)$, we have
$\frac{\sin t_0}{t_0} = y < \frac{3\sqrt{3}}{4\pi} = \frac{\sin(2\pi/3)}{2\pi/3}$.
Since $\frac{\sin t}{t}$ is decreasing on $(0,\pi)$, it follows that $t_0>\frac{2\pi}{3} > 2$.
\end{proof}

\begin{lemma}\label{lm:gamma}
For every $p \in (0,1)$, we have $(1-p)(2-p)\Gamma\left(\frac{1-p}{2}\right) \geq 2^{\frac{3-p}{2}}$.
\end{lemma}
\begin{proof}
Thanks to the claim from Lemma \ref{lm:c2-cinf}, it suffices to check the stated inequality at the endpoints: for $p = 0$, it becomes $2\sqrt{\pi} \geq 2^{3/2}$ which clearly holds, whereas for $p \to 1-$, it becomes equality.
\end{proof}

The following lemma is an important step in the proof of \eqref{eq:ball-int}. Essentially it is a consequence of convexity of sums of exponential functions.

\begin{lemma}\label{lm:F/G}
For $p \in (0,1)$ and $m = 1, 2, \dots$, we set
\[
R_m(p) = \frac{5}{3}\Big(\pi^{-1/2}2^{1/2-p}\Big)\frac{\Big(\log \big(\pi(m+3/2)\big)\Big)^{1-p/2}}{\pi(m+3/2)} \left(2^p + 2\pi^p\sum_{k=1}^m k^p\right).
\]
We have, $R_m(p) > 1$.
\end{lemma}
\begin{proof}
For $m = 1,2,\dots$, we let
\begin{align*}
A_m &= \frac{5}{3}\sqrt{\frac{2}{\pi}}\frac{\log\big(\pi(m+3/2)\big)}{\pi(m+3/2)},\\
a_{0,m} &= \Big(\log\big(\pi(m+3/2)\big)\Big)^{-1/2}, \\
a_{k,m} &= \frac{\pi k}{2}\Big(\log\big(\pi(m+3/2)\big)\Big)^{-1/2}, \quad k = 1, 2, \dots, m.
\end{align*}
Then
\[
R_m(p) = A_m\left(a_{0,m}^p + 2\sum_{k=1}^m a_{k,m}^p\right),
\]
which is a sum of convex functions, thus $R_m(p)$ is convex.

\emph{Case $m = 1$.} We have, $R_1'(p) \leq R_1'(1) = A_1(a_{0,1}\log a_{0,1} + 2a_{1,1}\log a_{1,1}) < -0.019$, so $R_1$ is decreasing on $(0,1)$. Thus for every $p \in (0,1)$, we have $R_1(p) \geq R_1(1) = A_1(a_{0,1}+2a_{1,1}) > 1.006$, as desired.

\emph{Case $m \geq 2$.} We have,
\[
R_m'(0) = A_m\left(\log a_{0,m} + 2\sum_{k=1}^m\log a_{k,m}\right) = A_m\log\left(a_{0,m}\prod_{k=1}^m a_{k,m}^2\right)
\]
and
\[
b_m = a_{0,m}\prod_{k=1}^m a_{k,m}^2 = \left(\frac{\pi}{2}\right)^{2m}(m!)^2\Big(\log\big(\pi(m+3/2)\big)\Big)^{-1/2-m}.
\]
We check directly that $b_2 > 2.7$ and $b_3 > 17$. For $m \geq 4$, we use the standard estimate $m! > \sqrt{2\pi m}\left(\frac{m}{e}\right)^m$ and $\log\big(\pi(m+3/2)\big) = \log\frac{3\pi}{2} + \log(1 + \frac{2}{3}m) < 2 + \frac{2m}{3} \leq \frac{m}{2} + \frac{2m}{3} = \frac{7m}{6}$ to obtain
\[
b_m > \frac{2\pi m}{\sqrt{\frac{7m}{6}}}\left(\frac{\pi^2}{4e^2}\frac{m^2}{\frac{7m}{6}}\right)^m > \left(\frac{6}{7}\frac{\pi^2}{e^2}\right)^m > 1.1^m.
\]
Therefore, $R_m'(0) > 0$ for every $m \geq 2$ and, by convexity, $R_m(p)$ is increasing. Thus,
\[
R_m(p) \geq R_m(0) = A_m (1+2m) = \frac{5}{3\pi}\sqrt{\frac{2}{\pi}}\log\big(\pi(m+3/2)\big)\frac{2m+1}{m+3/2}.
\]
For $m \geq 2$, the right hand side is lower bounded by its value at $m=2$, which is greater than $1.4$.
\end{proof}

The final lemma in this section lies at the heart of the base case of the inductive argument from Step II (Section \ref{sec:large}).

\begin{lemma}\label{lm:xp}
For $x \in [0,1]$ and $p \in [0,2]$, let
\[
h(p,x) = \left(\frac{1+x}{2}\right)^{2-p}-\left(\frac{1-x}{2}\right)^{2-p} + x\left(\frac{3-x^2}{2}\right)^{-p/2}.
\]
Then for every $x \in (0,1)$, $p \mapsto h(p,x)$ is strictly concave and decreasing on $[0,2]$. In particular, $h(p,x) \leq h(0,x) = 2x$, for every $x \in [0,1]$, $p \in [0,2]$.
\end{lemma}
\begin{proof}
First we show concavity. Fix $x \in (0,1)$. We have,
\begin{align*}
\frac{\partial^2}{\partial p^2}h(p,x) = \left(\frac{1+x}{2}\right)^{2-p}\log^2\frac{1+x}{2} &-\left(\frac{1-x}{2}\right)^{2-p}\log^2\frac{1-x}{2} \\
&+ \frac{x}{4}\left(\frac{3-x^2}{2}\right)^{-p/2}\log^2\frac{3-x^2}{2}.
\end{align*}
Then $(\frac{1-x}{2})^{p}\frac{\partial^2}{\partial p^2}h(p,x)$ is a strictly convex function of $p$ as being of the form $Aa^p+Bb^p-C$ with positive $a, A, b, B, C$. Therefore, in order to show that $\frac{\partial^2}{\partial p^2}h(p,x)$ is negative for $p \in (0,2)$, it suffices to check that it is nonpositive at the endpoints $p = 0$ and $p = 2$.

At $p = 0$, using $0 \leq \log(1+t) \leq t$, $t \geq 0$, we have
\begin{align*}
\frac{\partial^2}{\partial p^2}h(p,x)|_{p=0} &= \left(\frac{1+x}{2}\right)^{2}\log^2\frac{1+x}{2} -\left(\frac{1-x}{2}\right)^{2}\log^2\frac{1-x}{2} + \frac{x}{4}\log^2\frac{3-x^2}{2} \\
&\leq \left(\frac{1+x}{2}\right)^{2}\log^2\frac{1+x}{2} -\left(\frac{1-x}{2}\right)^{2}\log^2\frac{1-x}{2} + \frac{x}{4}\left(\frac{1-x^2}{2}\right)^2.
\end{align*}
Let $f(t) = \left(\frac{\log t}{1-t}\right)^2$, $t \in (0,1)$. With $a = \frac{1-x}{2}$ and $b = \frac{1+x}{2}$, then the right hand side can be written as $(ab)^2\big(f(b) - f(a) + b-a\big)$. Note that the power-series expansion of $f(1-t) = \left(\frac{\log (1-t)}{t}\right)^2 = \left(\sum_{k=0}^\infty \frac{t^k}{k+1}\right)^2$ has all the coefficients positive. In particular, $f$ is convex on $(0,1)$. Moreover, a direct computation shows that $\lim_{t \to 1-} f'(t) = -1$. Thus, $f(b) - f(a) \leq f'(b)(b-a) \leq -(b-a)$, which gives $\frac{\partial^2}{\partial p^2}h(p,x)|_{p=0} \leq 0$, as desired.

At $p = 2$, using $\log^2\frac{3-x^2}{2} < 1$ and $3-x^2 > 2$, we have
\begin{align*}
\frac{\partial^2}{\partial p^2}h(p,x)|_{p=2} &= \log^2\frac{1+x}{2} -\log^2\frac{1-x}{2} + \frac{x}{2(3-x^2)}\log^2\frac{3-x^2}{2} \\
&< \log^2\frac{1+x}{2} -\log^2\frac{1-x}{2} + \frac{x}{4}.
\end{align*}
Note that the right hand side at $x=0$ is $0$, so it suffices to show that it is decreasing in $x$. The derivative of the right hand side equals $\frac{2}{1+x}\log\frac{1+x}{2} + \frac{2}{1-x}\log\frac{1-x}{2} + \frac{1}{4}$. With the aid of $\log t \leq t-1$, $t > 0$, we upper bound this by $\frac{9}{4} -2\left(\frac{1}{1+x}+\frac{1}{1-x}\right) = 4\left(\frac{9}{16} - \frac{1}{1-x^2}\right) < 0$. Thus, $\frac{\partial^2}{\partial p^2}h(p,x)|_{p=2} < 0$, as desired. This finishes the proof of the concavity of $p \mapsto h(p,x)$.

To show that $p \mapsto h(p,x)$ is decreasing on $[0,2]$, since it is concave, it suffices to show that $\frac{\partial}{\partial p}h(p,x)|_{p=0} \leq 0$. We let
\[
g(x) = \frac{\partial}{\partial p}h(p,x)|_{p=0} = -\left(\frac{1+x}{2}\right)^{2}\log\frac{1+x}{2} +\left(\frac{1-x}{2}\right)^{2}\log\frac{1-x}{2} - \frac{x}{2}\log\frac{3-x^2}{2}.
\]
The rest of the argument is a tedious analysis of the derivatives of $g$, which we only sketch. Since $g'''(x) = \frac{2x^2(x^4-14x^2+9)}{(1-x^2)(3-x^2)^3}$ and $x^4-14x^2+9$ on $[0,1]$ changes sign only once from positive to negative, we get that $g''(x)$ on $(0,1)$ is first increasing and then decreasing. Moreover, $g''(0+) = 0$ and $g''(1-) = -\infty$. Thus $g''(x)$ on $(0,1)$ changes sign only once from positive to negative. Therefore, $g'(x)$ on $(0,1)$ is first increasing and then decreasing. Since $g'(0+) = -\frac12 + \frac12\log\frac83 < 0$ and $g'(1-) = 0$, we infer that $g'(x)$ on $(0,1)$ changes sign only once from negative to positive. Thus $g(x)$ on $(0,1)$ first decreases and then increases. Since $g(0+) = 0 = g(1-)$, we get that $g(x) < 0$ on $(0,1)$, which finishes the proof.
\end{proof}

\section{Proofs}

\subsection{Fourier-analytic formula}

The following important Fourier-analytic formula for negative moments is the starting point of our proof.

\begin{lemma}[Lemma 3 in \cite{GF}]\label{lm:formula-mom}
For a random vector $X$ in $\R^d$ and $p \in (0,d)$, we have
\[
\E\|X\|^{-p} = b_{p,d}\int_{\R^d}\phi_X(t)\|t\|^{p-d} \dd t,
\]
provided that the right hand side integral exists, where $\phi_X(t) = \E e^{i\scal{t}{X}}$ is the characteristic function of $X$, $\|\cdot\|$ is the Euclidean norm on $\R^d$ and $b_{p,d} = 2^{-p}\pi^{-d/2}\frac{\Gamma\left((d-p)/2\right)}{\Gamma(p/2)}$.
\end{lemma}

Using this formula, we have
\begin{equation}\label{eq:formula-mom}
\E\left|\sum_{k=1}^n a_kU_k\right|^{-p} = b_{p,1}\int_{\R}\left(\prod_{k=1}^n\frac{\sin(a_kt)}{a_kt}\right)|t|^{p-1} \dd t = 2b_{p,1}\int_0^\infty \left(\prod_{k=1}^n\frac{\sin(a_kt)}{a_kt}\right)t^{p-1} \dd t.
\end{equation}

The proof proceeds using completely different arguments depending on whether there is a \emph{large} weight $a_k$ or not.

\subsection{All weights are small}\label{sec:small}

Our goal here is the following special case of Theorem \ref{thm:khin-unif}.

\begin{theorem}\label{thm:all-small}
For every $p \in (0,1)$, every $n \geq 1$ and every reals $a_1, \ldots, a_n$ such that $\max_{k \leq n} |a_k| \leq \frac{1}{\sqrt{2}}(\sum_{k=1}^n a_k^2)^{1/2}$, inequality \eqref{eq:khin-unif} holds.
\end{theorem}

For the proof, we can assume that $\sum_{k=1}^n a_k^2 = 1$ and by symmetry, additionally, that each $a_k$ is positive. Thus in this case $0 < a_k \leq \frac{1}{\sqrt{2}}$ for every $k$. Recall \eqref{eq:formula-mom}. By H\"older's inequality, since $\sum a_k^2 = 1$,
\begin{align}
\notag\int_0^\infty \left(\prod_{k=1}^n\frac{\sin(a_kt)}{a_kt}\right)t^{p-1} \dd t &\leq \prod_{k=1}^n \left(\int_0^\infty \left|\frac{\sin(a_kt)}{a_kt}\right|^{1/a_k^2}t^{p-1}\dd t\right)^{a_k^2}\\
&= \prod_{k=1}^n\left(\Psi_p(1/a_k^2)\right)^{a_k^2},\label{eq:holder}
\end{align}
where we define
\begin{equation}\label{eq:def-psi}
\Psi_p(s) = \int_0^\infty \left|\frac{\sin (t/\sqrt{s})}{t/\sqrt{s}}\right|^st^{p-1}\dd t.
\end{equation}
The next step is to maximize $\Psi_p(s)$ over $s \geq 2$. The answer varies depending on the value of $p$ and is given by either $s=2$ or $s \to \infty$.

\begin{lemma}\label{lm:psi-large-p}
Let $p \in (p_0,1)$. For every $s \geq 2$, we have $\Psi_p(s) \leq \Psi_p(2)$. Moreover, $\Psi_p(2) = (2b_{p,1})^{-1}c_2(p)$.
\end{lemma}

\begin{lemma}\label{lm:psi-small-p}
Let $p \in (0,p_0)$. For every $s \geq 2$, we have $\Psi_p(s) \leq \lim_{s'\to\infty} \Psi_p(s').$ Moreover, $\lim_{s'\to\infty} \Psi_p(s') = (2b_{p,1})^{-1}c_\infty(p)$.
\end{lemma}

Taking these lemmas for granted for a moment, we can finish the proof as follows. Suppose that $p \in (p_0,1)$. Then combining \eqref{eq:formula-mom}, \eqref{eq:holder} and Lemma \ref{lm:psi-large-p}, we have
\[
\E\left|\sum_{k=1}^n a_kU_k\right|^{-p} \leq 2b_{p,1}\prod_{k=1}^n \Big((2b_{p,1})^{-1}c_2(p)\Big)^{a_k^2} = c_2(p),
\]
obtaining ``half'' of \eqref{eq:khin-unif}, that is when $C_p = c_2(p)$. Of course, we proceed identically for $p \in (0,p_0)$ using Lemma \ref{lm:psi-small-p} to obtain the other half. Therefore, to finish the proof of Theorem \ref{thm:all-small}, it remains to prove Lemmas \ref{lm:psi-large-p} and \ref{lm:psi-small-p}.

\begin{proof}[Proof of Lemma \ref{lm:psi-large-p}]
Recalling \eqref{eq:def-psi}, the definition of $\Psi_p$, by a change of variables, the inequality $\Psi_p(s) \leq \Psi_p(2)$ is equivalent to
\[
\int_0^\infty \left|\frac{\sin t}{t}\right|^st^{p-1}\dd t \leq s^{-p/2}\Psi_p(2),
\]
which can be thought of as a Ball's integral inequality with the weight $t^{p-1}$ (Ball's inequality corresponds to the case $p = 1$, see \cite{Ball, NP}).
For the proof, we rewrite the right hand side as $\int_0^\infty g_p(t)^pt^{p-1}\dd t$ with a Gaussian function
\[
g_p(t) = \exp(-\sigma_p^2 t^2)
\]
for $\sigma_p>0$ defined such that for every $s$,
\[
s^{-p/2}\sigma_p^{-p}\frac{\Gamma(p/2)}{2} = \int_0^\infty g_p(t)^st^{p-1}\dd t = s^{-p/2}\Psi_p(2).
\]
We emphasize that this identity holds for every $s$ with $\sigma_p$ depending only on $p$ and that this is why the Gaussian function $g_p$ is a good function to compare $\frac{\sin t}{t}$ with. Our goal is then to show that
\begin{equation}\label{eq:int-goal}
\int_0^\infty \left|\frac{\sin t}{t}\right|^st^{p-1}\dd t \leq \int_0^\infty g_p(t)^st^{p-1}\dd t, \qquad s \geq 2,
\end{equation}
and $\sigma_p$ in the definition of $g_p$ is such that there is equality for $s = 2$ in \eqref{eq:int-goal}. We remark that the equality for $s=2$ is equivalent to
\[
\E\left|U_1+U_2\right|^{-p} = \E|2\sigma_pZ|^{-p},
\]
where $U_1, U_2$ are i.i.d. uniform $[-1,1]$ random variables and $Z$ is a standard Gaussian random variable (because $\left(\frac{\sin t}{t}\right)^2$ is the characteristic function of $U_1+U_2$ and $g_p(t)^2$ is the characteristic function of $2\sigma_pZ$). This allows to explicitly compute $\sigma_p$,
\[
\sigma_p^p =2^{-p}\frac{\E|Z|^{-p}}{\E|U_1+U_2|^{-p}} = 2^{-1-p/2}\pi^{-1/2}(1-p)(2-p)\Gamma\left(\frac{1-p}{2}\right).
\]
To prove \eqref{eq:int-goal}, we use the following ``lemma on distribution functions'' from \cite{NP}. Recall that given a non-negative function $h: X\to [0,+\infty)$ on a measure space $(X,\mu)$ its \textit{distribution function} is the non-increasing function $H:(0,+\infty)\to [0,\infty)$ defined by
\[
H(y) := \mu(\{x\in X : h(x)>y\}).
\]
\begin{lemma}[\cite{NP}]\label{lm:NP.distr}
Let $f$ and $g$ be two non-negative measurable functions on a measure space $(X,\mu)$ and $F$, $G$ be the distribution functions of $f$ and $g$ respectively. If $F(y)$ and $G(y)$ are finite for every $y>0$ and there is some point $y_0$ such that $F(y)\gr G(y)$ for all $0 < y < y_0$ and $F(y)\ls G(y)$ for all $y>y_0$, then the function
\[
s\mapsto \frac{1}{sy_0^s}\int_X (f^s-g^s)\,d\mu
\]
is decreasing on the set $S=\{s>0 : f^s-g^s\in L^1(X,\mu)\}$.

In particular, if $\int_X (f^{s_0}-g^{s_0}) \,d\mu = 0$ for some $s_0>0$, then $\int_X (f^s-g^s) \,d\mu \ls 0$ for every $s\gr s_0$.
\end{lemma}
Let $\mu$ be the Borel measure on $(0,+\infty)$ with $\dd \mu(t) = t^{p-1} \dd t$. Let $F$ and $G_p$ be the distribution functions respectively of
\[
f(t) = \left|\frac{\sin t}{t}\right| \qquad \text{and} \qquad g_p(t) = \exp(-\sigma_p^2t^2).
\]
By Lemma \ref{lm:NP.distr}, to establish the validity of \eqref{eq:int-goal} it suffices to show that the difference $F-G_p$ changes sign on $(0,1)$ exactly once (since both $f$ and $g_p$ are bounded by $1$, both $F$ and $G_p$ vanish on $[1,+\infty)$). Notice that for $t \in (0,\pi)$,
\[
f(t) = \prod_{k=1}^\infty \left(1-\frac{t^2}{\pi^2k^2}\right) \leq \prod_{k=1}^\infty e^{-\frac{t^2}{\pi^2k^2}} = e^{-t^2/6}.
\]
Moreover, thanks to Lemma \ref{lm:c2-cinf}, our assumption $p \in (p_0,1)$ is equivalent to $\E\left|\frac{Z}{\sqrt{3}}\right|^{-p} < \E\left|\frac{U_1+U_2}{\sqrt{2}}\right|^{-p}$ which in turn by the definition of $\sigma_p$ is equivalent to $\sigma_p^2 < \frac{1}{6}$. Thus, $f(t) \leq e^{-t^2/6} < e^{-\sigma_p^2t^2} = g_p(t)$ for $t \in (0,\pi)$. Consequently,
\begin{equation}\label{eq:FGp-large-y}
F(y) < G_p(y) \qquad  \text{for } y \in (y_1,1),
\end{equation}
where $y_m = \max_{t \in [m\pi, (m+1)\pi]} f(t)$, $m = 1, 2, \ldots$ is the decreasing sequence of successive maxima of $f$, as in \cite{NP}. Since $
\int_0^\infty 2y(F(y) - G_p(y)) \dd y = \int_0^\infty (f(t)^2 - g_p(t)^2) \dd\mu(t)$,
we have that $F-G_p$ changes its sign at least once on $(0,1)$. Therefore, to prove that this happens exactly once, it suffices to prove that $G_p - F$ is strictly increasing on $(0,y_1)$, and since $G_p'$ and $F'$ are negative, equivalently that $|F'| > |G_p'|$ on every interval $(y_m, y_{m+1})$, $m \geq 1$.

To this end, fix an integer $m \geq 1$ and $y \in (y_{m+1}, y_m)$. Note that there is one solution, call it $t_0=t_0(y)$, to the equation $f(t) = y$ on $(0,\pi)$ and for every $k = 1,\ldots, m$, there are two solutions $t_{k}^-=t_{k}^-(y)$ and $t_{k}^+=t_{k}^+(y)$ on $(k\pi,(k+1)\pi)$. We can then write
\begin{align*}
F(y) &= \int_0^\infty \1_{[f(t)>y]}(t)\,t^{p-1}\,dt\\
     &= \sum_{k=0}^m \int_{k\pi}^{(k+1)\pi}\1_{[f(t)>y]}(t)\,t^{p-1}\,dt = \int_0^{t_0} t^{p-1}\,dt + \sum_{k=1}^m \int_{t_k^-}^{t_k^+}t^{p-1}\,dt.
\end{align*}
Differentiating with respect to $y$ we get
\[
F'(y) = t_0^{p-1}t_0'+\sum_{k=1}^m (t_k^+)^{p-1}(t_k^+)' - (t_k^-)^{p-1}(t_k^-)',
\]
so that
\[
|F'(y)| = \sum_{\{t > 0: f(t) = y\}} \frac{t^{p-1}}{|f'(t)|}.
\]
With the aid of Lemma \ref{lm:cos-sin} we then have
\[
|F'(y)| \geq \frac{5}{6}\sum_{\{t > 0: f(t)=y\}} t^p.
\]
Since $t_{k}^\pm \geq k\pi$ for every $k\gr 1$ and, by Lemma \ref{lm:t0}, $t_0 > 2$ it follows that
\[
|F'(y)| \geq \frac{5}{6}\left(2^p + 2\pi^p\sum_{k=1}^m k^p\right).
\]
We remark that this estimate is valid for all $p \in (0,1)$.

For $G_p$, we have
\[
G_p(y) = \mu\left(\left(0, \sigma_p^{-1}\sqrt{-\log y}\right)\right) = \frac{(-\log y)^{p/2}}{p\sigma_p^p},
\]
thus
\[
\frac{1}{|G_p'(y)|} =2\sigma_p^py(-\log y)^{1-p/2}.
\]
Since $y(-\log y)^{1-p/2}$ is increasing if $y < e^{p/2-1}$ and $y_1 < \pi^{-1} < e^{-1} < e^{p/2-1}$, it is in particular increasing on $(0,y_1)$, so using $y > y_{m+1} \geq \left|\frac{\sin(\pi(m+3/2)}{\pi(m+3/2)}\right| = \frac{1}{\pi(m+3/2)}$, we have
\begin{equation}\label{eq:F/G}
\frac{|F'(y)|}{|G_p'(y)|} \geq \frac{5}{3}\sigma_p^p\frac{\Big(\log \big(\pi(m+3/2)\big)\Big)^{1-p/2}}{\pi(m+3/2)} \left(2^p + 2\pi^p\sum_{k=1}^m k^p\right).
\end{equation}
Note that this estimate holds for all $p \in (0,1)$.
Applying Lemma \ref{lm:gamma} to lower bound $\sigma_p^p$ by $\pi^{-1/2}2^{1/2-p}$ and then Lemma \ref{lm:F/G} to lower bound the whole expression by $1$ finishes the proof.
\end{proof}

\begin{proof}[Proof of Lemma \ref{lm:psi-small-p}]
Finding the limit
\begin{equation}\label{eq:lim-Psi}
\lim_{s \to \infty} \Psi_p(s) = (2b_{p,1})^{-1}c_\infty(p)
\end{equation}
is standard. For instance, if the limit is taken along integral even $s$, this follows from Lemma \ref{lm:formula-mom} combined with the central limit theorem. In general, a simple analytic argument goes as follows: letting $K_s(t) = \left|\frac{\sin (t/\sqrt{s})}{t/\sqrt{s}}\right|^s$, splitting the integration as
\[
\Psi_p(s) = \int_0^\infty K_s(t)t^{p-1}\dd t = \int_0^{\pi\sqrt{s}} K_s(t)t^{p-1}\dd t + \int_{\pi\sqrt{s}}^\infty K_s(t)t^{p-1}\dd t
\]
and using $K_s(t) \leq \left|\frac{\sqrt{s}}{t}\right|^s$, we see the second integral is bounded by $s^{s/2}\int_{\pi\sqrt{s}}^\infty t^{p-1-s} \dd t = \frac{\pi^ps^{p/2}}{(s-p)\pi^s}$ which goes to $0$ as $s\to\infty$. Since $K_s(t)\1_{[0,\pi\sqrt{s}]}(t) \leq e^{-t^2/6}$ and in fact pointwise $\lim_{s\to\infty} K_s(t) = \lim_{s\to\infty} \left|1-\frac{t^2}{6s} + o\left(\frac{t^2}{s}\right)\right|^s = e^{-t^2/6}$, we obtain \eqref{eq:lim-Psi} from the first integral by Lebesgue's dominated convergence theorem.

Fix $s \geq 2$. By a change of variables, $\Psi_p(s) \leq (2b_{p,1})^{-1}c_\infty(p)$ can be rewritten as
\begin{equation}\label{eq:int-ineq-psmall}
\int_0^\infty \left|\frac{\sin t}{t}\right|^s t^{p-1}\dd t \leq \int_0^\infty \exp(-st^2/6)t^{p-1}\dd t.
\end{equation}
From this point onwards, we repeat the proof of Lemma \ref{lm:psi-large-p} with $f(t) = \left|\frac{\sin t}{t}\right|$ and $g(t) = \exp(-t^2/6)$, so $g(t) = g_p(t)$ with $\sigma_p$ set to be constant, equal to $\frac{1}{\sqrt{6}}$. For $s=2$ inequality \eqref{eq:int-ineq-psmall} is equivalent to $c_2(p) \leq c_\infty(p)$ which holds true and is in fact a strict inequality for every $p \in (0,p_0)$ (Lemma \ref{lm:c2-cinf}). We next look at the sign changes of the difference $F-G$ of the distribution functions $F$, $G$ of $f$ and $g$, respectively. If there is no sign change, we are immediately done (in view of the identity $\int (f^s - g^s) \dd \mu = \int_0^\infty sy^{s-1}(F(y)-G(y)) \dd y$). Thus, in view of Lemma \ref{lm:NP.distr}, it remains to check that $F-G$ changes sign at most once. Since \eqref{eq:FGp-large-y} holds here as well (with $G_p$ replaced by $G$), as in Lemma \ref{lm:psi-large-p}, it suffices to check that $|F'| > |G'|$ on every interval $(y_m,y_{m+1})$, $m \geq 1$. As in the proof of Lemma \ref{lm:psi-large-p}, we have inequality \eqref{eq:F/G} with $\sigma_p^p$ replaced by $6^{-p/2}$. Since $6^{-p/2} > \pi^{-1/2}2^{1/2-p}$ for every $p \in (0,1)$, Lemma \ref{lm:F/G} allows to finish the proof.
\end{proof}

\subsection{There is a large weight}\label{sec:large}

We finish the proof of Theorem \ref{thm:khin-unif} by following the inductive approach from \cite{NP}. It crucially relies on strengthening the right hand side of \eqref{eq:khin-unif} to allow the induction on the number of summands $n$ to work. To this end, we define
\[
\phi_p(x) = (1+x)^{-p/2}
\]
and
\[
\Phi_p(x) = \begin{cases} \phi_p(x), & x \geq 1, \\ 2\phi_p(1)-\phi_p(2-x), & 0 \leq x \leq 1.\end{cases}
\]
By this construction, the graph of $\Phi_p(x)$ on $[0,1]$ is the graph of $\phi_p(x)$ on $[1,2]$ reflected about the point $(1,\phi_p(1))$. In particular, to the left of $x = 1$, $\Phi_p$ and $\phi_p$ share the common tangent line at $x = 1$. Consequently, $\Phi_p(x) \leq \phi_p(x)$ for every $x$. By homogeneity, \eqref{eq:khin-unif} is equivalent to
\[
\E\left|U_1 + \sum_{k=2}^n a_kU_k\right|^{-p} \leq C_p\phi_p\left(\sum_{k=2}^n a_k^2\right).
\]
We shall inductively show a strengthening. As it will be clear from the proof, it is natural to run the inductive argument for spherically symmetric random vectors $\xi_k$.

\begin{theorem}\label{thm:one-large}
For every $p \in (0,1)$, every $n \geq 2$ and every vectors $v_2, \ldots, v_n$ in $\R^3$, we have
\begin{equation}\label{eq:ind}
\E\left|\scal{e_1}{\xi_1} + \sum_{k=2}^n \scal{v_k}{\xi_k}\right|^{-p} \leq C_p\Phi_p\left(\sum_{k=2}^n \|v_k\|^2\right).
\end{equation}
\end{theorem}
Since $\scal{v_k}{\xi_k}$ has the same distribution as $\|v_k\|U_k$, \eqref{eq:ind} gives \eqref{eq:khin-unif}.

\begin{proof}[Proof of Theorem \ref{thm:one-large}]
We use induction on $n$. For $n=2$, we have the following lemma, the proof of which we defer for now.

\begin{lemma}\label{lm:base}
For every vector $v$ in $\R^3$, we have
\begin{equation}\label{eq:base}
\E|\scal{e_1}{\xi_1} + \scal{v}{\xi_2}|^{-p} \leq c_2(p)\Phi_p(\|v\|^2).
\end{equation}
\end{lemma}

Let $n \geq 3$ and suppose \eqref{eq:ind} holds for every sequence of $n-1$ vectors in $\R^3$. Let $v_2, \dots, v_n \in \R^3$ and $x = \|v_2\|^2+\dots+\|v_n\|^2$. We want to show \eqref{eq:ind}. There are 3 cases.

\emph{Case (a):} $\|v_k\| > 1$ for some $2 \leq k \leq n$. Then $x > 1$, so \eqref{eq:ind} coincides with
\begin{equation}\label{eq:K''}
\E\left|\sum_{k=1}^n \scal{v_k}{\xi_k}\right|^{-p} \leq C_p(\|v_1\|^2+\|v_2\|^2+\ldots+\|v_n\|^2)^{-p/2},
\end{equation}
where $v_1 = e_1$.
Let $v_1^\ast,\ldots,v_n^\ast$ be a rearrangement of $v_1,\ldots,v_n$ such that $\|v_k^\ast\|\geq \|v_{k+1}^\ast\|$ for every $k=1,\ldots,n-1$ and let $v'_k = \frac{v_k^\ast}{\|v_1^\ast\|}$ for every $k=1,\ldots,n$, so that $\|v'_1\|=1$ and $\|v'_k\|\leq 1$ for $k=2,\ldots,n$. Then due to the homogeneity of \eqref{eq:K''} and the fact that $\scal{v'_1}{\xi_1}$ has the same distribution as $\scal{e_1}{\xi_1}$, it is enough to prove
\[
\E\left|\langle e_1,\xi_1\rangle + \sum_{k=2}^n \langle v'_k\xi_k\rangle\right|^p \ls C_p\Phi_p(\|v'_2\|^2+\dots+\|v'_n\|^2),
\]
which is handled by the next cases.

\emph{Case (b):} $\|v_k\| \leq 1$ for every $2 \leq k \leq n$ and $x \geq 1$. Then again \eqref{eq:ind} coincides with the homogeneous estimate \eqref{eq:khin-unif}. Moreover, we have that
\[
\max_{1\leq k\leq n} \|v_k\|^2=1\leq \frac{1}{2}(1+x) = \frac{1}{2}\sum_{k=1}^n \|v_k\|^2,
\]
so this case reduces to Theorem \ref{thm:all-small} where all the $\|v_k\|$ are \emph{small}.

\emph{Case (c):} $\|v_k\| \leq 1$ for every $2 \leq k \leq n$ and $x < 1$. Since the pair $(\xi_{n-1},\xi_n)$ has the same distribution as $(\xi_{n-1},Q\xi_{n-1})$ for a random orthogonal matrix $Q$ independent of all the $\xi_k$, we have,
\begin{align*}
&\E\left|\scal{e_1}{\xi_1} + \sum_{k=2}^n \scal{v_k}{\xi_k}\right|^p = \E\left|\scal{e_1}{\xi_1} + \scal{v_2}{\xi_2} + \dots + \scal{v_{n-1}}{\xi_{n-1}}+ \scal{Q^\top v_n}{\xi_{n-1}}\right|^p \\
&= \E_{Q}\left[\E_{(\xi_k)_{k=2}^{n-1}}\left|\scal{e_1}{\xi_1} +  \scal{v_2}{\xi_2} + \dots + \scal{v_{n-2}}{\xi_{n-2}} + \scal{v_{n-1}+Q^\top v_n}{\xi_{n-1}} \right|^p\right].
\end{align*}
By the inductive hypothesis applied to the sequence $(v_2, \dots, v_{n-2}, v_{n-1}+Q^\top v_n)$ (conditioned on the value of $Q$), we get
\[
\E\left|\scal{e_1}{\xi_1} + \sum_{k=2}^n \scal{v_k}{\xi_k}\right|^p \leq C_p\E_{Q}\Phi_p(\|v_2\|^2 + \dots + \|v_{n-2}\|^2 + \|v_{n-1}+Q^\top v_n\|^2).
\]
Note that
\[
\|v_2\|^2 + \dots + \|v_{n-2}\|^2 + \|v_{n-1}+Q^\top v_n\|^2 = x + 2\scal{v_{n-1}}{Q^\top v_n},
\]
thus, by the symmetry of $Q$,
\begin{align*}
&\E_{Q}\Phi_p(\|v_2\|^2 + \dots + \|v_{n-2}\|^2 + \|v_{n-1}+Q^\top v_n\|^2) \\
&= \E_{Q}\frac{\Phi_p\left(x + 2\scal{v_{n-1}}{Q^\top v_n}\right) + \Phi_p\left(x - 2\scal{v_{n-1}}{Q^\top v_n}\right)}{2}.
\end{align*}
We shall now need a lemma about concavity of $\Phi_p$, the proof of which we also defer.
\begin{lemma}\label{lm:Phi-concave}
Let $p \in (0,1)$. For every $a, b \geq 0$ with $\frac{a+b}{2} \leq 1$, we have
\[
\frac{\Phi_p(a)+\Phi_p(b)}{2} \leq \Phi_p\left(\frac{a+b}{2}\right).
\]
\end{lemma}
This lemma applied to $a = x + 2\scal{v_{n-1}}{Q^\top v_n}$ and $b = x - 2\scal{v_{n-1}}{Q^\top v_n}$ (which satisfy $a, b \geq 0$ and $\frac{a+b}{2} = x < 1$) finishes the proof of the inductive step.

It remains to show the lemmas we have used.
\end{proof}

\begin{proof}[Proof of Lemma \ref{lm:base}]
First note that if $\|v\|>1$ then, due to rotational invariance,
\begin{align*}
\mathbb{E}|\langle e_1,\xi_1\rangle + \langle v,\xi_2\rangle|^{-p} &= \|v\|^{-p} \mathbb{E}\left|\left\langle \frac{e_1}{\|v\|},\xi_1\right\rangle + \left\langle \frac{v}{\|v\|},\xi_2\right\rangle\right|^{-p}\\
                                                                   &= \|v\|^{-p} \mathbb{E}|\langle v',\xi_1\rangle + \langle e_1,\xi_2\rangle|^{-p}
\end{align*}
for any $v'\in\mathbb{R}^3$ such that $\|v'\| = \frac{1}{\|v\|}<1$, while at the same time
\[
\Phi_p(\|v\|^2) = \phi_p(\|v\|^2) = \|v\|^{-p}\phi_p(\|v'\|^2).
\]
This shows that the desired inequality is then equivalent to
\[
\mathbb{E}|\langle e_1,\xi_1\rangle + \langle v',\xi_2\rangle|^{-p}\ls c_2(p)\phi_p(\|v'\|^2), \quad \|v'\|\leq 1.
\]
Since $\Phi_p(x) \leq \phi_p(x)$ for $x \in [0,1]$, it is sufficient to prove the lemma in the case $\|v\| \leq 1$.

Fix $v \in \R^3$ with $x=\|v\|\ls 1$. To compute explicitly the left hand side of \eqref{eq:base}, recall that for any $w\in\mathbb{R}^3$, $\langle w,\xi\rangle$ has the same distribution as $\|w\|U$ where $\xi$ and $U$ are uniformly distributed on $S^2$ and $[-1,1]$, respectively. Then, we have that
\begin{align*}
\mathbb{E}|\langle e_1,\xi_1\rangle + \langle v,\xi_2\rangle|^{-p} = \mathbb{E}|U_1+xU_2|^{-p} &= \frac{1}{4}\int_{-1}^1\int_{-1}^1 |u_1+xu_2|^{-p} \dd u_1\dd u_2 \\
&=\frac{(1+x)^{2-p}-(1-x)^{2-p}}{2(1-p)(2-p)x}.
\end{align*}
Recalling the definition of $c_2(p)$ and $\Phi_p$ on $[0,1]$, we thus get that \eqref{eq:base} becomes
\[
\frac{(1+x)^{2-p}-(1-x)^{2-p}}{2x} \ls 2^{1-p/2}(2^{1-p/2}-(3-x^2)^{-p/2})
\]
for every $0<x\ls 1$. Note that we can write this as
\[
\left(\frac{1+x}{2}\right)^{2-p}-\left(\frac{1-x}{2}\right)^{2-p} + x\left(\frac{3-x^2}{2}\right)^{-p/2} \ls 2x,
\]
so Lemma \ref{lm:xp} finishes the proof.
\end{proof}

\begin{proof}[Proof of Lemma \ref{lm:Phi-concave}]
We can assume without loss of generality that $a<b$. If $b\leq 1$, the desired inequality follows from the concavity of $\Phi_p$ on $[0,1]$. So, assume that $b>1$. Then using the facts $a<b$ and $\frac{a+b}{2}\ls 1$, we can write
\[
\frac{\partial}{\partial a}\Phi_p\left(\frac{a+b}{2} \right) = \frac{1}{2}\frac{d\Phi_p}{dx}\Big|_{x=\frac{a+b}{2}} \ls \frac{1}{2} \frac{d\Phi_p}{dx}\Big|_{x=a} = \frac{\partial}{\partial a} \frac{\Phi_p(a)+\Phi_p(b)}{2},
\]
using the fact that the derivative of $\Phi_p'$ is decreasing on $[0,1]$. This implies that
\[
\Phi_p\left(\frac{a+b}{2} \right) - \frac{\Phi_p(a)+\Phi_p(b)}{2}
\]
is a decreasing function of $a$, so to prove the desired inequality, it suffices to show that the latter is nonnegative for the maximum value of $a$, that is $a = a_0=2-b$. Since $b>1$, $a_0 < 1$ and by the definition of $\Phi_p$, $\Phi_p(a_0) = 2\phi_p(1) - \phi_p(2-(2-b)) = 2\phi_p(1) - \phi_p(b)$ and $\Phi_p(b) = \phi_p(b)$, we get
\[
\frac{\Phi_p(a_0)+\Phi_p(b)}{2} = \frac{2\phi_p(1)-\phi_p(b)+\phi_p(b)}{2} = \phi_p(1) = \Phi_p\left(\frac{a_0+b}{2}\right),
\]
that is, the desired inequality is in fact an equality in this case.
\end{proof}

\begin{remark}\label{rem:proj}
Let $p \in (0,1)$. Let $X$ be a rotationally invariant random vector in $\R^3$. For every nonzero vector $y$ in $\R^3$, observing that $\|X+y\|^2$ has the same distribution as $\|X\|^2 + \|y\|^2 + 2\|X\|\|y\|U$, where $U$ is uniform on $[-1,1]$, independent of $X$, we have
\[
\E\|X+y\|^{-p} = \E\frac{\big(\|X\|+\|y\|\big)^{2-p}-\big|\|X\|-\|y\|\big|^{2-p}}{2(2-p)\|X\|\|y\|}.
\]
In particular, by the concavity of $t \mapsto t^{1-p}$, 
\[
\E\|X+y\|^{-p} \leq \E\|X\|^{-p}.
\]
This combined with independence gives
\[
\E\left\|\sum_{k=1}^n a_k\xi_k\right\|^{-p} \leq \min_{1 \leq k \leq n} |a_k|^{-p}.
\]
For $p=1$, this immediately gives \eqref{eq:L-1-1} in the case of a large weight, $\max_{k \leq n} |a_k| > \frac{1}{\sqrt{2}}\left(\sum_{k=1}^n a_k^2\right)^{1/2}$ and Theorem \ref{thm:one-large} is not needed (this corresponds to the simple projection argument from \cite{Ball} handling this case). For $p < 1$, this argument yields the nonsharp constant $2^{p/2}$ instead of $(1-p)C_p$.
\end{remark}

\subsection{Proof of Corollary \ref{cor:khin-sph}}

Let $G=(G_1,G_2,G_3)$ be a standard Gaussian random vector in $\R^3$ (mean $0$, covariance~$I$), independent of the sequence $(\xi_k)_{k=1}^n$. Then for every vector $x$ in $\R^3$, since $\scal{x}{G}$ has the same distribution as $\|x\|G_1$, we have $\|x\|^{-p} = \alpha_p\E|\scal{x}{G}|^{-p}$ with $\alpha_p = (\E|G_1|^{-p})^{-1}$. Therefore,
\begin{equation}\label{eq:emb}
\left\|\sum_{k=1}^n A_k\xi_k\right\|^{-p} = \alpha_p\E_G\left|\sum_{k=1}^n\scal{\xi_k}{A_k^\top G}\right|^{-p}.
\end{equation}
Using this and inequality \eqref{eq:khin-unif}, we obtain
\begin{align*}
\E\left\|\sum_{k=1}^n A_k\xi_k\right\|^{-p} &= \alpha_p\E_G\E_{\xi}\left|\sum_{k=1}^n\scal{\xi_k}{A_k^\top G}\right|^{-p} \leq C_p\alpha_p\E_G\left(\sum_{k=1}^n \|A_k^\top G\|^2\right)^{-p/2}.
\end{align*}
Rewriting the sum of squares using the second moment, applying Minkowski's inequality (with the negative exponent $-\frac{2}{p}$) and using \eqref{eq:emb} again, we get
\begin{align*}
C_p\alpha_p\E_G\left(\sum_{k=1}^n \|A_k^\top G\|^2\right)^{-p/2} &= C_p\alpha_p\E_G\left(3\E_\xi \left|\sum_{k=1}^n \scal{\xi_k}{A_k^\top G}\right|^2\right)^{-p/2} \\
&\leq 3^{-p/2}C_p\alpha_p\left(\E_\xi\left(\E_G\left|\sum_{k=1}^n \scal{\xi_k}{A_k^\top G}\right|^{-p}\right)^{-2/p}\right)^{-p/2}\\
&= 3^{-p/2}C_p \left(\E_\xi\left\|\sum_{k=1}^n A_k\xi_k\right\|^2\right)^{-p/2}.
\end{align*}
Finally, $3\E_\xi\left\|\sum_{k=1}^n A_k\xi_k\right\|^2 = \sum_{k=1}^n\|A_k\|_{HS}^2$.
\hfill$\square$

\section{Conclusion}

Continuing a long line of work  and particularly addressing some questions raised in \cite{BC}, we have established a sharp Khinchin-type $L_{-p}-L_2$ moment comparison inequality when $p \in (0,1)$ for weighted sums of independent random variables uniform on $[-1,1]$, equivalently uniform vectors on the unit sphere $S^2$ in $\R^3$. In this case, this provides a sharp version of the very general results from \cite{GF}.

We have not tried to optimise various technical numerical estimates which would certainly allow to extend our results to $p \in (0,p_1)$, $p_1 =1.38$ (the negative moments of order $-p$ for $S^2$--uniform vectors exist for all $p < 2$). The arguments seem robust enough to handle cases of $S^r$--uniform vectors for other values of $r$ (most notably the case of $r = 1$ corresponding to Steinhaus random variables as well as the case of $r=3$ which would provide extensions of the polydisc slicing inequality of Oleszkiewicz and Pe\l czy\'nski from \cite{OP}, just as our result extends Ball's cube slicing inequality from \cite{Ball}). Moreover, the question of a sharp $L_p-L_2$, $p \in (0,1)$, moment comparison for $\sum_{k=1}^n a_kU_k$ remains open with a natural conjecture that $\inf_n \inf_{a \in \R^n} \frac{\E|\sum_{k=1}^n a_kU_k|^p}{\left(\sum_{k=1}^n a_k^2\right)^{p/2}} = \lim_{n\to\infty} \E\left|\sum_{k=1}^nn^{-1/2}U_k\right|^p$ (see Question 5 and Proposition 15 in \cite{ENT}). All this is the topic of ongoing and future work.


\begin{thebibliography}{9}



\bibitem{AH}
Averkamp, R., Houdr\'e, C., Wavelet thresholding for non-necessarily Gaussian noise: Idealism. \emph{Ann. Statist.} 31 (2003), 110--151.


\bibitem{BC}
Baernstein, A., II, Culverhouse, R., Majorization of sequences, sharp vector Khinchin inequalities, and bisubharmonic functions. \emph{Studia Math.} 152 (2002), no. 3, 231--248.

\bibitem{Ball}
Ball, K.,
Cube slicing in $R^n$.
\emph{Proc. Amer. Math. Soc.} 97 (1986), no. 3, 465--473.



\bibitem{Ball2}
Ball, K., Volumes of sections of cubes and related problems. \emph{Geometric aspects of functional analysis (1987--88)}, 251--260, Lecture Notes in Math., 1376, \emph{Springer, Berlin}, 1989.

\bibitem{Ball3}
Ball, K.,
Mahler's conjecture and wavelets.
\emph{Discrete Comput. Geom.} 13 (1995), no. 3-4, 271--277.


\bibitem{BN}
Barthe, F.; Naor, A. 
Hyperplane projections of the unit ball of $l^n_p$. 
\emph{Discrete Comput. Geom.} 27 (2002), no. 2, 215--226.



\bibitem{Eat}
Eaton, M. L.,
A note on symmetric Bernoulli random variables.
\emph{Ann. Math. Statist.} 41 (1970), 1223–-1226.



\bibitem{EM}
Edmunds, D., Melkonian, H.,
Behaviour of $L_q$ norms of the $\text{sinc}_p$ function.
\emph{Proc. Amer. Math. Soc.} 147 (2019), no. 1, 229--238.



\bibitem{ENT1}
Eskenazis, A., Nayar, P., Tkocz, T.,
Gaussian mixtures: entropy and geometric inequalities, \emph{Ann. of Prob.} 46(5) 2018, 2908--2945.


\bibitem{ENT}
Eskenazis, A., Nayar, P., Tkocz, T.,
Sharp comparison of moments and the log-concave moment problem.
\emph{Adv. Math.} 334 (2018), 389--416.



\bibitem{FHJSZ}
Figiel, T., Hitczenko, P., Johnson, W. B., Schechtman, G., Zinn, J.,
Extremal properties of Rademacher functions with applications to the Khintchine and Rosenthal inequalities.
\emph{Trans. Amer. Math. Soc.} 349 (1997), no. 3, 997--1027.


\bibitem{GF}
Gorin, A., Favorov, Yu.,
Generalizations of the Khinchin inequality. (Russian)
\emph{Teor. Veroyatnost. i Primenen.} 35 (1990), no. 4, 762--767; translation in
\emph{Theory Probab. Appl.} 35 (1990), no. 4, 766--771 (1991).



\bibitem{Haa}
Haagerup, U.,
The best constants in the Khintchine inequality.
\emph{Studia Math.} 70 (1981), no. 3, 231--283.

\bibitem{Ha}
Hadwiger, H.
Gitterperiodische Punktmengen und Isoperimetrie.
\emph{Monatsh. Math.} 76 (1972), 410--418.


\bibitem{HT}
Havrilla, A., Tkocz, T., Sharp Khinchin-type inequalities for symmetric discrete uniform random variables, preprint, arXiv:1912.13345.



\bibitem{H}
Hensley, D.,
Slicing the cube in $R^n$ and probability (bounds for the measure of a central cube slice in $R^n$ by probability methods).
\emph{Proc. Amer. Math. Soc.} 73 (1979), no. 1, 95--100.


\bibitem{KOS}
Kerman, R., Olhava, R., Spektor, S.,
An asymptotically sharp form of Ball's integral inequality.
\emph{Proc. Amer. Math. Soc.} 143 (2015), no. 9, 3839--3846.


\bibitem{Kh}
Khintchine, A.,
\"Uber dyadische Br\"uche.
\emph{Math. Z.} 18 (1923), no. 1, 109--116.


\bibitem{Kom}
Komorowski, R.,
On the best possible constants in the Khintchine inequality for $p\geq3$.
\emph{Bull. London Math. Soc.} 20 (1988), no. 1, 73--75.


\bibitem{K}
K\"onig, H.,
On the best constants in the Khintchine inequality for Steinhaus variables.
\emph{Israel J. Math.} 203 (2014), no. 1, 23--57.


\bibitem{KKol}
K\"onig, H., Koldobsky, A.,
On the maximal measure of sections of the $n$-cube. \emph{Geometric analysis, mathematical relativity, and nonlinear partial differential equations}, 123--155,
Contemp. Math., 599, \emph{Amer. Math. Soc., Providence, RI}, 2013.


\bibitem{KKol2}
König, H., Koldobsky, A.,
On the maximal perimeter of sections of the cube.
\emph{Adv. Math.} 346 (2019), 773--804.


\bibitem{KK}
K\"onig, H., Kwapie\'n, S.,
Best Khintchine type inequalities for sums of independent, rotationally invariant random vectors.
\emph{Positivity} 5 (2001), no. 2, 115--152.



\bibitem{KR}
K\"onig, H., Rudelson, M.,
On the volume of non-central sections of a cube.
\emph{Adv. Math.} 360 (2020), 106929, 30 pp.




\bibitem{KLO}
Kwapie\'n, S., Lata\l a, R., Oleszkiewicz, K.,
Comparison of moments of sums of independent random variables and differential inequalities.
\emph{J. Funct. Anal.} 136 (1996), no. 1, 258--268.



\bibitem{LO-best}
Lata\l a, R., Oleszkiewicz, K.,
On the best constant in the Khinchin-Kahane inequality.
\emph{Studia Math.} 109 (1994), no. 1, 101--104.




\bibitem{LO}
Lata\l a, R., Oleszkiewicz, K.,
A note on sums of independent uniformly distributed random variables.
\emph{Colloq. Math.} 68 (1995), no. 2, 197--206.



\bibitem{LPP}
Livshyts, G., Paouris, G., Pivovarov, P.,
On sharp bounds for marginal densities of product measures. 
\emph{Israel J. Math.} 216 (2016), no. 2, 877--889.


\bibitem{Mo}
Mordhorst, O.,
The optimal constants in Khintchine's inequality for the case $2<p<3$.
\emph{Colloq. Math.} 147 (2017), no. 2, 203--216.



\bibitem{NO}
Nayar, P., Oleszkiewicz, K.,
Khinchine type inequalities with optimal constants via ultra log-concavity.
\emph{Positivity} 16 (2012), no. 2, 359--371.



\bibitem{NP}
Nazarov, F., Podkorytov, A.,
Ball, Haagerup, and distribution functions. Complex analysis, operators, and related topics, 247--267,
\emph{Oper. Theory Adv. Appl.}, 113, Birkh\"auser, Basel, 2000.



\bibitem{OP}
Oleszkiewicz, K., Pe\l czy\'nski, A.,
Polydisc slicing in $C^n$.
\emph{Studia Math.} 142 (2000), no. 3, 281--294.


\bibitem{Pe}
Pe\v{s}kir, G.,
Best constants in Kahane-Khintchine inequalities for complex Steinhaus functions.
\emph{Proc. Amer. Math. Soc.} 123 (1995), no. 10, 3101--3111.


\bibitem{Sa}
Sawa, J.,
The best constant in the Khintchine inequality for complex Steinhaus variables, the case $p=1$. \emph{Studia Math.} 81 (1985), no. 1, 107--126.


\bibitem{Sz}
Szarek, S.,
On the best constant in the Khintchine inequality.
\emph{Stud. Math.} 58, 197--208 (1976).





\end{thebibliography}
\end{document}